\documentclass[11pt]{amsart}

\usepackage{amsthm,amsmath,verbatim,amssymb}
\newtheorem{lem}{Lemma}

\newtheorem{thm}{Theorem}

\title{Periodic solutions of Abreu's equation}

\author{Renjie Feng and G\'abor Sz\'ekelyhidi\textsuperscript{\dag}}
\thanks{\dag\,Partially supported by NSF grant DMS-0904223}
\address{Department of Mathematics, Northwestern University, Evanston,
IL}
\email{renjie@math.northwestern.edu}
\address{Department of Mathematics, Columbia University, New York, NY}
\email{gabor@math.columbia.edu}

\begin{document}
\begin{abstract}
	We solve Abreu's equation with periodic right hand side, in any
	dimension. This can be interpreted as prescribing the scalar
	curvature of a torus invariant metric on an Abelian variety.
\end{abstract}
\maketitle

\section{Introduction}

Given a smooth periodic function $A:\mathbb{R}^n\to\mathbb{R}$, we would
like to find a convex function $u:\mathbb{R}^n\to\mathbb{R}$ such that
\begin{equation}\label{eq:abreu}
	\sum_{i,j} \frac{\partial^2 u^{ij}}{\partial x_i\partial x_j} = A,
\end{equation}
where $u^{ij}$ is the inverse of the Hessian of $u$. By an affine
transformation we can assume that the fundamental domain for the
periodicity is $\Omega=[0,1]^n$. Our main result is the following.

\begin{thm}\label{thm:main}
	For any smooth periodic $A$ with mean $0$,
	we can find a smooth periodic
	function $\phi:\mathbb{R}^n\to\mathbb{R}$ such that
	\begin{equation}\label{eq:phi}
		u(x) = \frac{1}{2}|x|^2 + \phi(x)
	\end{equation}
	is a convex
	solution of Equation (\ref{eq:abreu}). Moreover $\phi$ is
	unique up to adding a constant.
\end{thm}

\noindent The proof shows that the result is true with $\frac{1}{2}|x|^2$
replaced with any strictly convex smooth function $f$ for which $f_{ij}$
is periodic. 
That $A$ must have mean $0$ can be seen by a simple integration by
parts.

The Equation (\ref{eq:abreu}) was introduced by Abreu~\cite{A}
in the study of the scalar curvature of toric varieties. In that case
the domain of $u$ is a convex polytope in $\mathbb{R}^n$, and $u$ is
required to have prescribed boundary behaviour near the boundary of
the polytope, obtained by Guillemin~\cite{G}. This equation has been
studied extensively by Donaldson~\cite{ D3, D, D4} finally solving it in
the 2 dimensional case, when $A$ is a constant in \cite{D5}. The main
difficulty is dealing with the boundary of the polytope, and in our case
this does not arise. As a result our problem is significantly simpler
and the ideas in \cite{D} and also Trudinger-Wang~\cite{TW} can be used
to solve the equation in any dimension. 
Other recent works on Abreu's equation include
the works of Chen-Li-Sheng~\cite{CLS1, CLS2}.

We solve the equation using the continuity method, which relies on
finding a priori estimates for the solution $u$. We first set up the
continuity method, and prove the uniqueness of the solution
in Section~\ref{sec:cont}. Then we need to
obtain estimates for the determinant of the Hessian $\det(u_{ij})$ from
above and below, which can be done by a maximum principle argument
adapted from \cite{TW}. We will show this in Section~\ref{sec:det}.
Then in Section~\ref{sec:higher},
following \cite{TW} we use the theorem of
Caffarelli-Gutierrez~\cite{CG} to obtain H\"older estimates for
$\det(u_{ij})$, from which higher order estimates follow from
Caffarelli~\cite{C} and the Schauder estimates. This will complete the
proof of Theorem~\ref{thm:main}. An important step in
this method is to obtain the strict convexity of the solution $u$ in the
sense that we need to control at least one section of the convex
function $u$. For
our problem this is quite easy to show in any dimension even
independently of the equation, but in
\cite{TW} and \cite{D} this is the step which restricts the results to
$n=2$.

We mentioned that Abreu's equation on a convex polytope with suitable
boundary data is related to prescribing the scalar curvature on toric
varieties. Similarly periodic solutions are related to prescribing the
scalar curvature on Abelian varieties. We will explain this in
Section~\ref{sec:abelian}.

\subsection*{Acknowledgements}
The first named author would like to thank his advisor S. Zelditch for
his support and encouragement, and also V. Tosatti for many helpful discussions.

\section{The continuity method and uniqueness}\label{sec:cont}
Given any smooth
periodic $A$ with average $0$, we use the continuity method to solve the
family of equations
\begin{equation}\label{eq:cont}
	(u^{ij})_{ij} = tA,
\end{equation}
where $u(x) =\frac{1}{2}|x|^2 + \phi(x)$ for some periodic
function $\phi$ and we
are summing over the repeated indices.
	
Let us write $S\subset [0,1]$ for the set of parameters $t$ for
which we can solve Equation~(\ref{eq:cont}). Clearly $0\in S$, since then
$\phi=0$ is a solution, thus $S$ is nonempty. To
show $S$ is open, we use the implicit function theorem.

\begin{lem}
	The set $S$ is open.
\end{lem}
\begin{proof}
Since we are only interested in periodic $\phi$ it is natural to work on
the torus $T^n$. Let us write $C^{k,\alpha}_0(T^n)$ for the space of
$C^{k,\alpha}$ functions on $T^n$ with average 0. Define the map
\[ \begin{aligned}
	\mathcal{F} : C^{4,\alpha}_0(T^n) &\to C^{0,\alpha}_0(T^n)  \\
	\phi &\mapsto (u^{ij})_{ij},
\end{aligned}\]
where $u(x)=\frac{1}{2}|x|^2 + \phi(x)$ as before. Note that if $\phi$
is periodic
then so is $(u^{ij})_{ij}$ and also integration by parts shows that
$(u^{ij})_{ij}$ has average 0.

Consider the
linearization $\mathcal{L}$ of $\mathcal F$ at $\phi$
which is given by
\[ \mathcal{L}(\psi) =(u^{ia}\psi_{ab}u^{bj})_{ij}.\] Then
$\mathcal L$ gives a linear elliptic operator
\[ \mathcal L:C_0^{4,\alpha}(T^n)\rightarrow C^{0,\alpha}_0(T^n).\]
The operator is self-adjoint and its kernel is trivial so it is an
isomorphism. By the implicit function theorem, this means
that if Equation~(\ref{eq:cont}) has a smooth solution for $t=t_0$,
then we can solve the equation for all nearby $t$ too. This nearby
solution is a priori in $C^{4,\alpha}$, but by elliptic regularity it is
actually smooth. Hence $S$ is open.
\end{proof}

The fact that $S$ is closed follows from the a priori estimates for the
solution given in Lemma~\ref{lem:higher}, which will complete the
existence part of Theorem~\ref{thm:main}. For now we
will prove the uniqueness of the solution.

\begin{lem}
	Suppose that $\phi_0$ and $\phi_1$ are periodic and
	that $u=\frac{1}{2}|x|^2 + \phi_0$ and $v=\frac{1}{2}|x|^2 +
	\phi_1$ are convex
	functions satisfying
	\[ (u^{ij})_{ij} = (v^{ij})_{ij} = A.\]
	Then $\phi_0=\phi_1 + c$ for some constant $c$.
\end{lem}
\begin{proof}
Consider the functional
\[ \mathcal F_A(\phi)= \int_\Omega -\log \det(u_{ij}) +
A\phi\,d\mu,\]
where $\Omega$ is the fundamental domain, $u=\frac{1}{2}|x|^2 + \phi$
as usual and $d\mu$ is the Lebesgue measure.
This is analogous to the functional used by Donaldson~\cite{D3} which in
turn is based on the Mabuchi functional~\cite{M}.

This functional is convex along the linear path $\phi_t=(1-t)\phi_0 +
t\phi_1$, in fact writing $\psi=\phi_1-\phi_0$ we have
\begin{equation}\label{eq:convex}
	\frac{d^2}{dt^2} \mathcal{F}_A(\phi_t) = \int_\Omega
(u_t)^{ia}\psi_{ab}(u_t)^{bj}\psi_{ij}\,d\mu\geqslant 0,
\end{equation}
where $u_t = \frac{1}{2}|x|^2 + \phi_t$.

At the same time
\[ \frac{d}{dt} \mathcal{F}_A(\phi_t) = \int_\Omega -(u_t)^{ij}\psi_{ij} +
A\psi\,d\mu = \int_\Omega \big[A-(u_t^{\,ij})_{ij}\big]\psi\,d\mu,\]
where we can integrate by parts without a boundary term
because of the periodicity. By our assumptions, this derivative vanishes
for $t=0$ and $t=1$, so by the convexity (\ref{eq:convex}), the
functional $\mathcal{F}_A(\phi_t)$ is constant for $t\in[0,1]$. It then
follows from (\ref{eq:convex}) that
\[ \int_\Omega (u_0)^{ia}\psi_{ab}(u_0)^{bj}\psi_{ij}\,d\mu = 0, \]
and since $u_0$ is convex and $\psi$ is periodic, this implies that
$\psi$ is a constant.
\end{proof}

\section{Bounds for the determinant}\label{sec:det}
In this section we will prove the following.
\begin{lem}\label{lem:det}
	Suppose that $\phi:\mathbb{R}^n\to\mathbb{R}$ is a
	periodic smooth function and $u(x)=\frac{1}{2}|x|^2+\phi(x)$
	satisfies
	Abreu's Equation (\ref{eq:abreu}). Then
	\[ c_1<\det(u_{ij})<c_2,\]
	where $c_1,c_2 >0$ are constants only
	depending on $A$.
\end{lem}

First we give a simple $C^0$ and $C^1$ bound for $\phi$.
\begin{lem} \label{dds}
	If $\phi$ is periodic, $\phi(0)=0$ and $\frac{1}{2}|x|^2 +
	\phi(x)$ is convex, then
	\[ |\phi|, |\nabla\phi| < C\]
	for some constant $C$.
\end{lem}
\begin{proof}
	Since $\phi$ is periodic, it is enough to consider the
	fundamental domain $\Omega=[0,1]^n$.
	Since $\frac{1}{2}|x|^2+\phi$ is convex, we must
	have the lower bound $D^2\phi > -\mathrm{Id}$ on the Hessian of
	$\phi$.
	If $\sup_\Omega\phi = \phi(x_{max})$ then we get for every $y$
	that
	\[ \phi(y) > \phi(x_{max}) - |y-x_{max}|^2.\]
	Using that $\phi$ is periodic and $\phi(0)=0$, it follows from
	this that $|\phi|<C$ for some uniform $C$.
	It follows from the convexity of $u$ that
	$|\nabla\phi| < C$ too. 
\end{proof}

Now we turn to the proof of Lemma \ref{lem:det}, using ideas from
Trudinger-Wang~\cite{TW}. While it is not
strictly necessary, it is convenient to study instead the Legendre
transform $v$ of the convex function $u$. The dual coordinate $y$ is
defined by $y = \nabla u$,
and then $v$ is defined by the equation
\[ v(y) + u(x) = y\cdot x.\]
The Legendre transform  $v$ is of the form
\[ v(y) = \frac{1}{2}|y|^2 + \psi(y),\]
where $\psi:\mathbb{R}^n\to\mathbb{R}$ is periodic. Moreover $v$
satisfies the equation
\begin{equation}\label{eq:neweq}
	\begin{aligned}
		v^{ij}L_{ij} &= \tilde{A} \\
		L &= \log\det(v_{ab}),
	\end{aligned}
\end{equation}
where $\tilde{A}(y) = A(x)$, so certainly $\sup|\tilde{A}| = \sup|A|$.
Since $\det(v_{ab})(y) = \det(u_{ab})^{-1}(x)$, to prove
Lemma~\ref{lem:det} it is enough to bound $\det(v_{ab})$ from above and
below by a constant depending on $\sup|\tilde{A}|$.

\subsection{Upper bound for $\det(u_{ij})$}

Consider the function
\[ f(y)=L + \frac{1}{2}|y|^2+2\psi,\]
where $L=\log\det(v_{ab})$.
Since $v_{ij}$ and $\psi$ are both periodic,
the global minimum of $f$ must be achieved at some point
$p\in [-1,1]^n$. We can rewrite $f$ as
\[ f(y)=L - \frac{1}{2}|y|^2+2v. \]
At the point $p$ we have
\[ 0\leqslant v^{ij}f_{ij}=\tilde{A}-v^{ii}+2n, \]
So
\[ v^{ii}(p)\leqslant \sup|\tilde{A}| + 2n. \]
By the arithmetic-geometric mean inequality, we have
\[ \det (v^{ij})(p)\leqslant
\left(\frac{Tr(v^{ij})(p)}{n}\right)^n\leqslant c,\]
where $c=(\frac{1}{n}\sup_\Omega |\tilde{A}| +2)^n>0$ is a constant.
Since $p$ is the minimum of $f$, for any $y\in\Omega = [0,1]^n$ we have
\[ \log \det (v_{ij})(y)\geqslant \log \det
(v_{ij})(p)+\frac{1}{2}|p|^2+
2\psi(p)-\frac{1}{2}|x|^2-2\psi(x) > c',\]
for some constant $c'$. The last inequality is given since
$\det (v_{ij})(p)\geqslant \frac{1}{c}$ and the rest
can be bounded by Lemma \ref{dds} in the region $\Omega$.
Thus we have $\det (v_{ij})(x)> e^{c'}>0$. By Legendre duality, we have
$$\det(u_{ij})<e^{-c'}$$ where $c'$ is some constant only depending on
$\sup|\tilde{A}| = \sup|A|$.

\subsection{Lower bound for $\det(u_{ij})$}
Now consider the function
\[ g(y)=-L-\beta|\nabla v|^2+v = -L+\psi-\beta|\nabla
v|^2+\frac{1}{2}|y|^2.\]
Choose $\beta$ sufficiently small such that on  $\mathbb{R}^n$,
we have
\[ \beta |\nabla v|^2\leqslant \frac{1}{4}|y|^2+1.\]
Such $\beta$ exist since $\nabla v=y+\nabla \psi$ and $\sup|\nabla
\psi|<c$. This implies that
\[ \frac{1}{4}|y|^2 - 1 \leqslant -\beta|\nabla v|^2 + \frac{1}{2}|y|^2
\leqslant \frac{1}{2}|y|^2,\]
so that
\begin{equation}\label{eq:B2}
	\sup_{B(2)} \left(-\beta|\nabla v|^2 + \frac{1}{2}|y|^2 \right)
	<
	\inf_{\mathbb{R}^n\setminus B(4)}
	\left(-\beta|\nabla v|^2 + \frac{1}{2}|y|^2\right),
\end{equation}
where $B(R)$ denotes the ball of radius $R$ around the origin.
Since $ -L+\psi$ is periodic and the fundamental domain $\Omega\subset
B(2)$, this inequality shows that the global minimum of $g(y)$ is
achieved at a point $q\in B(4)$. Indeed if $y\not\in B(4)$ and
$y'\in\Omega$
denotes the corresponding point in the fundamental domain, then the
inequality (\ref{eq:B2}) implies that $g(y') < g(y)$, so the minimum
cannot be outside of $B(4)$.

At $q$, we have
\begin{equation}\label{eq:gi}
	0=g_i=-L_i-2\beta v_kv_{ki}+v_i
\end{equation}
and
\[ 0\leqslant v^{ij}g_{ij}=-\tilde{A}-
2\beta v^{ij}(v_{ki}v_{kj}+v_{k}v_{ijk})+n = -\tilde{A} - 2\beta v_{kk}
- 2\beta v_kL_k + n,\]
where we used that $v^{ij}v_{ijk}=L_{k}$.
Using Equation (\ref{eq:gi}) we then get that at the point $q$,
\[\begin{aligned}
	0&\leqslant -\tilde{A} - 2\beta v_{kk}-2\beta v_kL_k+n \\
	&= -\tilde{A}+n -2\beta v_{kk} - 2\beta|\nabla v|^2 +
		4\beta^2 v_k v_i v_{ik} \\
		&\leqslant \sup|\tilde{A}| +n- 2\beta v_{kk} + 4\beta ^2
		|\nabla v|^2 v_{kk}.
\end{aligned}\]
Because of Lemma \ref{dds}, we alreadly have a bound for $|\nabla v|$
on $B(4)$. So we can
choose $\beta$ sufficiently small such that $4\beta^2 |\nabla v|^2(q)
\leqslant \beta$, which implies
\[ \beta v_{kk}(q)\leqslant \sup|\tilde{A}| + n.\]
Thus we have at $q$, by arithmetic-geometric mean inequality again,
\[\log \det (v_{ij})(q)\leq c\]
for a constant $c$ depending only on $\sup|\tilde{A}|=\sup|A|$.
Since $q$ is the minimum of $g$ and we can bound $|\nabla v|$ and $|v|$
uniformly on $\Omega=[0,1]^n$ by Lemma \ref{dds}, we have for any
$y\in\Omega$
$$L(y)\leqslant L(q)+\beta\left(|\nabla v|^2(q)-|\nabla v|^2(y)\right)+
v(y)-v(q)<c''$$ where $c''$ is a finite constant only depending on
$\sup|A|$.
Thus by Legendre duality again, we have
\[ \det (u_{ij})=\det(v_{ij})^{-1} >e^{-c''} \]
on $\Omega$, but $\det(u_{ij})$ is periodic, so this lower bound holds
everywhere. This completes the proof of Lemma~\ref{lem:det}.

\section{Higher order estimates}\label{sec:higher}
Given the determinant bound of Lemma~\ref{lem:det},
we can obtain estimates for all
derivatives of the solution in terms of the ``modulus of convexity'',
using the results of Caffarelli-Gutierrez~\cite{CG} on solutions
of the linearized Monge-Amp\`ere equation and Caffarelli~\cite{C} on the
Monge-Amp\`ere equation. This is roughly
identical to the discussion in Section
5.1 of Donaldson~\cite{D}, which in turn is based on the work of
Trudinger and Wang~\cite{TW1,TW}.

\begin{lem}\label{lem:higher}
	Suppose that $\phi:\mathbb{R}^n\to\mathbb{R}$ is a periodic
	smooth function and $u(x)=\frac{1}{2}|x|^2 + \phi(x)$ satisfies
	Abreu's Equation~(\ref{eq:abreu}). Then there exist constants
	$c_k$ depending on $A$ such that
	\[ \begin{gathered}
		c_1 I < (u_{ij}) < c_2I \\
		|\nabla^k \phi| < c_k.
	\end{gathered}\]
\end{lem}

\begin{proof}
	To prove this, we rewrite the equation in the form
	\begin{equation}\label{eq:abreu2}
		\begin{aligned} U^{ij}w_{ij} &= A \\
				w &= (\det(u_{ab}))^{-1},
		\end{aligned}
	\end{equation}
	where $U^{ij}$ is the cofactor matrix of the Hessian $u_{ij}$.
	This form follows from the original form of the equation because
	$(U^{ij})_i=0$.

	Now define the section $S(h) = \{x\in \mathbb{R}^n\, |\, u(x) <
	h\}$. There exists a linear transformation $T$ which normalizes
	this section, which means
	\begin{equation}\label{eq:norm}
		B(0,\alpha_n) \subset T(S(h)) \subset B(0,1),
	\end{equation}
	where $\alpha_n>0$ is a constant depending on $n$ only. The bound
	on $\det(u_{ab})$ from Lemma~\ref{lem:det} and the results in
	\cite{CG} imply that we have a H\"older bound
	\[ \Vert w\Vert_{C^\alpha(S(h/2))} \leqslant C_0,\]
	but the constants $\alpha, C_0$ depend on the norms of $T$ and
	$T^{-1}$ (in addition to depending on $\sup w, \sup|A|$, but
	those are already controlled). In \cite{CG} this is shown for
	the homogeneous equation where $A=0$, but the argument can be
	extended to the non-homogeneous case. Namely the Harnack
	inequality (Theorem 5 in \cite{CG}) extends as explained in
	Trudinger-Wang~\cite{TW3} after which the arguments are
	identical.

	So we simply need to control $T$ and $T^{-1}$. Note that from
	Lemma~\ref{dds} we have
	\[ \frac{1}{2}|x|^2 - C \leqslant u(x) \leqslant
	\frac{1}{2}|x|^2 + C\]
	for some $C$, which implies that if $u(x)=h$, then
	\[ \sqrt{2(h-C)} \leqslant |x| \leqslant \sqrt{2(h+C)}.\]
	In particular for sufficiently large $h$ we have
	\begin{equation}\label{eq:sh}
		B(\sqrt{h}) \subset S(h) \subset B(\sqrt{3h}).
	\end{equation}
	Let us choose $h$ large enough so that $[0,1]^n\subset S(h/2)$.
	Then (\ref{eq:sh}) together with (\ref{eq:norm}) gives bounds on
	$T$ and $T^{-1}$ for this choice of $h$. We then have a
	$C^\alpha$ bound on $w$ in $S(h/2)$, and in particular in
	$[0,1]^n$. Since $w$ is periodic, this gives a $C^\alpha$ bound
	on $w$ everywhere.

	Now Caffarelli's Schauder estimate~\cite{C} gives $C^{2,\alpha}$
	bounds on $u$ in $[0,1]^n$. From this standard Schauder
	estimates~\cite{GT} applied to the Equations~(\ref{eq:abreu2})
	give bounds on the higher order derivatives of $u$.
\end{proof}

\section{Abelian varieties}\label{sec:abelian}
Theorem~\ref{thm:main}
can be interpreted as prescribing the scalar
curvature of a torus invariant metric in the K\"ahler class of a flat
metric over an Abelian variety. In this section we briefly explain this.

Let $V$ be a $n$-dimensional complex vector space and $\Lambda \cong
\mathbb Z^{2n}$ a maximal lattice in $V$ such that the quotient $M =
V/\Lambda$ is an Abelian variety, i.e., a complex torus which can be
holomorphically embedded in projective space.  For brevity, we only
consider the Abelian varieties $M=\mathbb{C}^n/\Lambda$ where
$\Lambda=\mathbb{Z}^n+i\mathbb{Z}^n$.
The case of general lattices can be reduced to this
case by an affine transformation.

We write $z\in M$ as $z=x+iy$,
where $x$ and $y \in \mathbb R^n$ and can be viewed as the periodic
coordinates of $M$. Let
\[\begin{aligned}\omega_0&= \frac{\sqrt{-1}}{2}
	\sum_{\alpha=1}^n dz_{\alpha}
\wedge d\bar z_{\alpha}= \sum_{\alpha=1}^n dx_{\alpha} \wedge d
y_{\alpha}\end{aligned}\]
be the standard flat metric with associated local K\"ahler potential
$\frac{1}{2}|z|^2$. The group $T^n=(S^1)^n$ acts on $M$ via translations
in the Langrangian subspace $i\mathbb R^n \subset \mathbb C^n$, thus we
can consider the following space of torus invariant K\"ahler metrics in
the fixed class $[\omega_0]$,
\[\mathcal H_{T^n}\ = \ \left\{\psi\in
C^{\infty}_{T^n} (M) : \omega_\psi\ = \ \omega_0+
\frac{\sqrt{-1}}{2}\partial \bar \partial \psi>0\right\}.\]
Functions
invariant under the translation of $T^n$ are independent of $y$, so
they are smooth function on $M/T^n\cong T^n$. In other words they are
smooth $\mathbb{Z}^n$-periodic functions in the variable $x\in
\mathbb R^n$.

Let us write the K\"ahler potential
in complex coordinates as
\[ f(z) = \frac{1}{2}|z|^2 + 4\psi(x),\]
where $\psi$ is a periodic function of $x$.
Then it is easy to see that $f_{i\bar j} = \delta_{ij} + \psi_{ij}$,
where $\psi_{ij}$ is the real Hessian of $\psi$. Let us also define the
function $v:\mathbb{R}^n\to\mathbb{R}$ given by
\[ v(x) = \frac{1}{2}|x|^2 + \psi(x).\]
Note that then $f_{i\bar j} = v_{ij}$ so in particular $v$ is strictly
convex.
It follows that the scalar
curvature of the metric is given by
\begin{equation}\label{s} S = -\sum_{i,\bar j} f^{i\bar j}\left[
	\log\det(f_{a\bar b})\right]_{i\bar j} =
-\frac{1}{4}\sum_{i,j} v^{ij}\left[ \log\det(v_{ab})\right]_{ij}.
\end{equation}

Let us now take the Legendre transform of $v$,
with dual coordinate $t=\nabla
v(x)$. The transformed function $u$ is defined by
\[ u(t) + v(x) = t\cdot x.\]
A calculation (see Abreu~\cite{A}) then gives
\begin{equation}\label{s2} 
	S(x) = -\frac{1}{4}\sum_{i,j} \frac{ \partial^2
	u^{ij}(t)}{\partial t_i\partial t_j},
\end{equation}
i.e.  the scalar curvature equation is equivalent to
 Abreu's equation, where the equivalence is given by the Legendre
transform. 

Thus Theorem~\ref{thm:main} implies that we can prescribe the
scalar curvature of torus invariant metrics on an Abelian variety, as
long as we work in the Legendre transformed ``symplectic'' coordinates
instead of the complex coordinates. For example in complex dimension 1,
working in complex coordinates essentially amounts to working in a fixed
conformal class. In this case Kazdan-Warner~\cite{KW} have given
necessary and sufficient conditions for a function to be the scalar
curvature of a metric conformal to a fixed metric $g$. 
Their conditions are that either
$S\equiv 0$, or the average of $S$ with respect to $g$
is negative and $S$ changes sign.
When working in symplectic coordinates (and only $S^1$-invariant
metrics) then we have seen that the condition is that $S$ has zero mean.

\end{document}